\documentclass[12pt, a4paper]{amsart}
\usepackage[bookmarksopen=true]{hyperref}

\usepackage[english]{babel}
\usepackage{amssymb}
\usepackage{mathrsfs}
\usepackage{dsfont}
\usepackage{amscd}
\usepackage{lmodern}
\usepackage{color}

\newtheorem{thmintro}{Theorem}

\newtheorem{corintro}[thmintro]{Corollary}
\newtheorem{conjintro}{Conjecture}
\newtheorem{theorem}{Theorem}[section]
\newtheorem{corollary}[theorem]{Corollary}
\newtheorem{lemma}[theorem]{Lemma}
\theoremstyle{definition}
\newtheorem{remarkintro}{Remark}
\newtheorem{remark}[theorem]{Remark}

\newcommand{\CC}{\mathbf{C}}
\newcommand{\NN}{\mathbf{N}}
\newcommand{\QQ}{\mathbf{Q}}
\newcommand{\RR}{\mathbf{R}}
\newcommand{\ZZ}{\mathbf{Z}}
\newcommand{\BBB}{\mathscr{B}}
\newcommand{\CCC}{\mathscr{C}}
\newcommand{\GGG}{\mathcal{G}}
\newcommand{\FFF}{\mathcal{F}}
\newcommand{\inv}{^{-1}}

\newcommand{\co}{\colon\thinspace}

\DeclareMathOperator{\Aut}{Aut}
\DeclareMathOperator{\rk}{rk}
\DeclareMathOperator{\CAT}{CAT(0)}
\DeclareMathOperator{\Min}{Min}
\DeclareMathOperator{\re}{re}
\DeclareMathOperator{\SL}{SL}
\DeclareMathOperator{\Hom}{Hom}
\DeclareMathOperator{\Lie}{Lie}
\DeclareMathOperator{\ad}{ad}
\DeclareMathOperator{\Res}{Res}
\DeclareMathOperator{\sph}{sph}
\DeclareMathOperator{\St}{St}
\DeclareMathOperator{\proj}{proj}
\DeclareMathOperator{\Conv}{Conv}
\DeclareMathOperator{\FA}{(FA)}
\DeclareMathOperator{\FB}{(FB)}
\DeclareMathOperator{\FBP}{(FB+)}

\begin{document}

\renewcommand{\proofname}{{\bf Proof}}

\title[A fixed point theorem for Lie groups acting on buildings]{A fixed point theorem for Lie groups acting on buildings and applications to Kac--Moody theory}

\author[Timoth\'ee~Marquis]{Timoth\'ee \textsc{Marquis}$^{*}$}
\address{UCL, 1348 Louvain-la-Neuve, Belgium}
\email{timothee.marquis@uclouvain.be}
\thanks{$^*$F.R.S.-FNRS Research Fellow}

%
\begin{abstract}
We establish a fixed point property for a certain class of locally compact groups, including almost connected Lie groups and compact groups of finite abelian width, which act by simplicial isometries on finite rank buildings with measurable stabilisers of points. As an application, we deduce amongst other things that all topological $1$-parameter subgroups of a real or complex Kac--Moody group are obtained by exponentiating $\ad$-locally finite elements of the corresponding Lie algebra. 
\end{abstract}

\maketitle

\section{Introduction}
Recall that a group $H$ is said to satisfy property $\FA$ if every action, without inversion, of $H$ on a simplicial tree has a fixed point. In this paper, we investigate some ``higher rank''  analog of this property, which we define in the class of locally compact groups. More precisely, we say that a locally compact group $G$ has {\bf property $\FB$} if it satisfies the following property:

\begin{itemize}
\item[$\FB$]
\emph{Every measurable action of $G$ on a finite rank building stabilises a spherical residue.}
\end{itemize}

Here, we mean by an {\bf action} of $G$ on a building $\Delta$ a type-preserving simplicial isometric action on $\Delta$. We call such an action {\bf measurable} if the stabilisers in $G$ of the spherical residues of $\Delta$ are Haar measurable. Note that continuous actions are examples of measurable actions. In particular every action of a discrete group is measurable. Thus in the special case of discrete groups, property $\FB$ is a direct analog of property $\FA$, where trees are replaced by arbitrary finite rank buildings. 

In this paper, however, we will focus on a class of non-discrete groups. More precisely, we prove the following theorem.

\begin{thmintro}\label{mainthm}
Let $G$ be an almost connected locally compact group. If $G$ has finite abelian width, then it has property $\FB$.
\end{thmintro}

Recall that a topological group $G$ is {\bf almost connected} if its group of components $G/G^0$ is compact. Also, a group $G$ is said to have {\bf finite abelian width} if there exist finitely many abelian subgroups $A_1,\dots,A_N$ of $G$ such that $G=\{a_1\dots a_N \ | \ a_i\in A_i\}$.

Since almost connected Lie groups have finite abelian width (see Remark~\ref{rk Lie faw} below), this implies in particular the following.
\begin{corintro}\label{corintro Lie}
Every almost connected Lie group has property $\FB$.
\end{corintro}

In another direction, compact groups of finite abelian width (and conjecturally all compact groups) also have property $\FB$.

\begin{corintro}\label{corintro compact}
Every compact group of finite abelian width has property $\FB$. In particular, compact $p$-adic analytic groups and profinite groups of polynomial subgroup growth have property $\FB$.
\end{corintro}

The proofs of Corollaries~\ref{corintro Lie} and \ref{corintro compact} will be given at the end of Section~\ref{section proof mainthm}.

Note that the ``finite abelian width'' hypothesis is not necessary in case the building is a tree.

\begin{thmintro}\label{thm tree}
Let $G$ be an almost connected locally compact group. Suppose that $G$ acts measurably by type-preserving simplicial isometries on a tree. Then $G$ has a global fixed point. 
\end{thmintro}

In fact, we conjecture that this hypothesis is also unnecessary in general.

\begin{conjintro}\label{conj measurable}
Every almost connected locally compact group has property $\FB$.
\end{conjintro}

More generally, we make the following conjecture, which would imply Conjecture~\ref{conj measurable} (see section~\ref{section rk cor}). Recall that an action of a group $G$ on a metric space $X$ is {\bf locally elliptic} if each $g\in G$ has a fixed point. 

\begin{conjintro}\label{conj tree}
Let $G$ be a group acting by type-preserving simplicial isometries on a building $\Delta$. If the $G$-action on the Davis realization $X$ of $\Delta$ is locally elliptic, then $G$ has a global fixed point in $X\cup\partial X$, where $\partial X$ denotes the visual boundary of $X$.
\end{conjintro}

Note that this conjecture is equivalent to asserting that, under the same hypotheses, $G$ fixes a point in the so-called {\bf combinatorial bordification} of $\Delta$, as we will see in Section~\ref{section rk cor}.

Before we state the corollaries to Theorem~\ref{mainthm} concerning Kac--Moody theory, we make some remarks about property $\FB$ itself.

\begin{remarkintro}\label{rkintro Serre}
The restriction to measurable actions in the statement of property $\FB$ is essential, since even the most basic examples of locally compact groups admit non-measurable actions on trees without fixed point. Indeed, if a group $G$ is the countable union of a strictly increasing sequence of subgroups, then by \cite[6.1 Thm.15]{Serre} it possesses an action without fixed point on a tree (and the stabilisers of points for this action are precisely the conjugates of the subgroups in the given sequence). Note that we can assume this tree to be a building, since we can glue rays at each endpoint of the tree without affecting the $G$-action.

We now construct such a sequence of proper subgroups for $G$ in case $G=(\RR,+)$. Let $\BBB$ be a basis of $\RR$ over $\QQ$ and let $\{x_i \ | \ i\in\NN\}\subset\BBB$ be a countable family of (pairwise distinct) basis elements. For each $n\in\NN$, let $V_n$ denote the $\QQ$-sub-vector space of $\RR$ with basis $\BBB\setminus\{x_i \ | \ i\geq n\}$. Then the additive groups of the $V_n$ yield the desired chain. Note that the existence of such a chain is equivalent to the existence of a non-Lebesgue measurable subset of $\RR$ (see Remark~\ref{rk positive measure} below), hence to the axiom of choice.

Remark also that for a nonzero $x$ in $V_{0}$, we may project this chain into $\RR/x\ZZ$, yielding an example of a compact group acting without fixed point on a tree. 
\end{remarkintro}

\begin{remarkintro}
In the context of almost connected locally compact groups, the notion of Haar measurability generalises both Borel and universal measurability. Indeed, since such a group $G$ is compactly generated, it is in particular $\sigma$-compact. Hence its Haar measure $\mu$ is $\sigma$-finite. Thus one can construct from $\mu$ a complete probability measure on $G$ whose measurable sets coincide with the Haar measurable sets. For this reason, we will only speak in this paper about Haar measurability.
\end{remarkintro}

\begin{remarkintro}
We finally remark that Conjecture~\ref{conj measurable} is not true in general for locally compact groups that are not almost connected. Indeed, note first that any group $G$ equipped with the discrete topology is locally compact. Moreover, any action of such a $G$ on a building is automatically measurable. In particular if $G$ does not have property $\FA$, then it does not have property $\FB$ either. Therefore, as far as property $\FB$ is concerned, one cannot hope to get any general statement about locally compact groups without some non-discreteness assumption.

Note also that if a locally compact group $G$ is equipped with a non-spherical BN-pair $(B,N)$ such that $B$ is open, then it acts measurably and strongly transitively on the associated non-spherical building, and hence does not satisfy property $\FB$. In particular, complete non-compact Kac--Moody groups over finite fields give another class of (totally disconnected non-compact) locally compact groups not satisfying property $\FB$. Indeed, we recall that such Kac--Moody groups are equipped with the compact open topology with respect to their action on the positive associated building. Hence stabilisers of chambers are open since the action is simplicial. For basics on Kac--Moody groups and their completions, we refer to \cite{CaRe} and \cite{theseBR}.
\end{remarkintro}

We now turn to the corollaries of Theorem~\ref{mainthm} concerning Kac--Moody groups over fields, as defined by Tits (\cite{Tits87}). The key result, which readily follows from Corollary~\ref{corintro Lie}, gives a partial answer to a problem stated in \cite[xi]{thesePE}.

\begin{thmintro}\label{thmintro KM}
Any measurable homomorphism of an almost connected Lie group into a Kac--Moody group has bounded image.
\end{thmintro}
Here, we equip a Kac--Moody group with the $\sigma$-algebra generated by the finite type parabolic subgroups of each sign; measurability in Lie groups (and more generally, in locally compact groups) will always be understood as Haar measurability.
We recall that a subgroup of a Kac--Moody group is {\bf bounded} if it is contained in the intersection of two finite type parabolic subgroups of
opposite signs.

\medskip
Let now $k$ be either $\RR$ or $\CC$. In the following statements, we consider \emph{adjoint} Kac-Moody groups $G=\GGG(k)$ over $k$, that is, images under the adjoint representation of Kac--Moody groups over $k$. In addition, we assume that they are generated by their root subgroups. 

Recall that such a $G$ is naturally endowed with a topology, the so-called Kac--Peterson topology, which turns it into a connected Hausdorff topological group (\cite{Kacpetersontopo}, \cite[Prop.5.15]{HartKohlmars}). A {\bf one-parameter subgroup} of $G$ is then a continuous homomorphism from $\RR$ to $G$. The set of all one-parameter subgroups of $G$ is denoted by $\Hom_c(\RR,G)$. As is well known, in case $G$ is a Lie group, this set can be given a Lie algebra structure so that it identifies with the Lie algebra of $G$ (see e.g. \cite[Prop.2.10]{HofMorliegr}). In fact, by the solution to Hilbert's fifth problem (see \cite{Montgomzippinhilbert5}), this construction extends to connected locally compact groups as well. The following result opens the way for analogs to these classical results within Kac--Moody theory. 

\begin{corintro}\label{cor KMR}
Every one-parameter subgroup $\alpha$ of a real or complex Kac--Moody group $G$ is of the form $\alpha(t)=\exp\ad(tx)$ for some $\ad$-locally finite $x\in\Lie(G)$.
\end{corintro}

We explain and prove this statement in section~\ref{section cor KM}.

Since maximal bounded subgroups can be given a Lie group structure (see Lemma~\ref{lemma cor KM CM} below), the automatic continuity of measurable homomorphisms between Lie groups (see e.g. \cite[Theorem 1]{automcont}) extends as follows: 

\begin{corintro}\label{corintro KM2}
Every measurable homomorphism between real or complex Kac--Moody groups is continuous.
\end{corintro}
Here by a {\bf measurable} homomorphism between two real or complex Kac--Moody groups, we mean a homomorphism $\phi\co G_1\to G_2$ between these groups such that the preimage of an open set of $G_2$ by the restriction of $\phi$ to any Lie subgroup of $G_1$ (that is, any closed subgroup of $G_1$ with a Lie group structure) is Haar measurable. Note that Borel homomorphisms are examples of measurable homomorphisms in this sense.

Finally, as a last consequence of Theorem~\ref{thmintro KM}, we get the following classification of measurable isomorphisms between real or complex Kac--Moody groups.

\begin{corintro}\label{corintro KM3}
Let $\alpha$ be a measurable isomorphism between real or complex Kac--Moody groups. Then $\alpha$ is continuous and standard, that is, it induces an isomorphism of the corresponding twin root data. 
\end{corintro}
Without using the assumption of measurability, this has been proved by Caprace (\cite{thesePE}). For almost split Kac--Moody groups of 2-spherical type, the corresponding result has been obtained by G. Hainke (\cite{Hainke}). Note however that our proof relying on Theorem~\ref{mainthm} is substantially shorter. 

We also refer to the end of Section~\ref{section cor KM} for detailed proofs of these corollaries.

\subsection*{Conventions} Throughout this paper, all buildings are assumed to have finite rank. Unless otherwise stated, we see buildings as simplicial complexes. 

\subsection*{Acknowledgments} I am very grateful to Pierre-Emmanuel Caprace for proposing this problem to me in the first place, as well as for various helpful comments and suggestions. I would also like to thank Tobias Hartnick and Ralf Köhl, as well as the anonymous referee for their useful comments.

\section{Property $\FB$}\label{section proof mainthm}

In this section, we establish the core of the argument for the proof of Theorem~\ref{mainthm} (see Theorem~\ref{thm K} below). We then deduce Corollary~\ref{corintro Lie} and \ref{corintro compact}.

\subsection*{Davis realization of a building}
Recall from \cite{Davis} that any building $\Delta$ admits a metric realization, here denoted by $X:=|\Delta|$, which is a complete $\CAT$ cell complex. Moreover any group of type-preserving automorphisms of $\Delta$ acts in a canonical way by cellular isometries on $X$. 
Finally, the cell supporting any point of $X$ determines a unique spherical residue of $\Delta$. In particular, an automorphism of $\Delta$ which fixes a point in $X$ must stabilise the corresponding spherical residue in $\Delta$.

\subsection*{The Bruhat--Tits fixed point theorem} 
Let $G$ be a group acting by isometries on a complete $\CAT$ space $X$. Then $G$ has a fixed point in $X$ if and only if its orbits in $X$ are bounded. Indeed, if $G$ fixes a point then it stabilises the spheres centered at that point. The converse follows from the Bruhat--Tits fixed point theorem (see for example \cite[Th.11.23]{BrownAbr}). We now state an easy application of this fact. 

We say that $G$ is a {\bf bounded product} of finitely many subgroups $U_1,\dots,U_n$, or that it is {\bf boundedly generated} by these subgroups, which we write $G=U_1\dots U_n$, if each element $g\in G$ can be written as $g=u_1\dots u_n$ for some $u_i\in U_i$, $1\leq i\leq n$. If $G$ is boundedly generated by abelian subgroups, we say that it has {\bf finite abelian width}. 

\begin{lemma}\label{lemma reduction}
$G$ fixes a point in $X$ as soon as one of the following holds:
\begin{itemize}
\item[(1)] $G$ is a bounded product of subgroups each fixing a point in $X$.
\item[(2)] There exists a finite-index subgroup of $G$ fixing a point in $X$.
\end{itemize} 
\end{lemma}
\begin{proof}
Suppose $G=U_1\dots U_n$ for some subgroups $U_i\leq G$ with bounded orbits in $X$. Then each $U_i$ maps a bounded set onto a bounded set. A straightforward induction now proves the first case. To prove the second case, let $H<G$ be a finite index subgroup of $G$ with bounded orbits in $X$. Write $G=\coprod_{i=1}^{n}{g_iH}$ for some $g_i\in G$. Then each $x\in X$ is mapped by $G$ onto the finite union $\bigcup_{i=1}^{n}{g_i(Hx)}$ of bounded subsets, which is again bounded.  
\end{proof}

\subsection*{Actions on $\CAT$ spaces}
Let $G$ be a group acting by isometries on a $\CAT$ space $X$. For every $g\in G$, we let $|g|:=\inf\{d(x,g\cdot x) \ | \ x\in X\}\in [0,\infty)$ denote its {\bf translation length} and we set $\Min(g):=\{x\in X \ | \ d(x,g\cdot x)=|g|\}$. An element $g\in G$ is said to be {\bf elliptic} if it fixes some point. For a subgroup $H\leq G$, we also write $\Min(H):=\bigcap_{h\in H}{\Min(h)}$ and we say that the $H$-action on $X$ is {\bf locally elliptic} if each $h\in H$ is elliptic.

For the convenience of the reader, we record the following result from \cite[Thm.1.1]{CL10}, which we will use in the course of the proof of Theorem~\ref{mainthm}.
\begin{lemma}\label{lemma CLytchack}
Let $X$ be a complete $\CAT$ space of finite geometric dimension and $\{X_{\alpha}\}_{\alpha\in A}$ be a filtering family of closed convex non-empty subspaces. Then either the intersection $\bigcap_{\alpha\in A}{X_{\alpha}}$ is non-empty, or the intersection of the visual boundaries $\bigcap_{\alpha\in A}{\partial X_{\alpha}}$ is a non-empty subset of $\partial X$ of intrinsic radius at most $\pi/2$.
\end{lemma}
Recall that a family of subsets $\FFF$ of a given set is called {\bf filtering} if for all $E,F\in\FFF$ there exists $D\in\FFF$ such that $D\subseteq E\cap F$. 
We point out that Davis realizations of buildings of finite rank (and closed convex subcomplexes) are examples of complete $\CAT$ spaces of finite geometric dimension (see \cite{Geomdim}), since these are finite dimensional $\CAT$ cell complexes.

\subsection*{Combinatorial bordification of a building}
We now recall some terminology introduced in \cite{CL11}. Let $\Delta$ be a building with Davis realization $X$, and let $G=\Aut_0(\Delta)$ be its group of type-preserving simplicial isometries. Let $\Res_{\sph}(\Delta)$ denote the set of spherical residues of $\Delta$. Note that, identifying a point of $X$ with its support, the $G$-actions on $X$ and $\Res_{\sph}(\Delta)$ coincide. 

The set $\Res_{\sph}(\Delta)$ can be turned into a metric space using the so-called {\bf root-distance} (\cite[1.2]{CL11}), whose restriction on the set of chambers coincide with the gallery distance. By looking at the projections of any given spherical residue $R$ onto all $\sigma\in\Res_{\sph}(\Delta)$, one gets a map \[ \pi_{\Res}\co \Res_{\sph}(\Delta)\to\prod_{\sigma\in\Res_{\sph}(\Delta)}{\St(\sigma)}: R\mapsto (\sigma\mapsto\proj_{\sigma}(R)),\]
where $\St(\sigma)$ denotes the {\bf star} of $\sigma$, that is, the set of all residues containing $\sigma$ in their boundaries. Endow the above product with the product topology, where each star is a discrete set of residues. The {\bf combinatorial bordification} of $\Delta$, denoted $\CCC_{\sph}(\Delta)$, is then defined as the closure of the image of $\pi_{\Res}$: $\CCC_{\sph}(\Delta)=\overline{\pi_{\Res}(\Res_{\sph}(\Delta))}$ (\cite[2.1]{CL11}). In the sequel, we will also write $\Res_{\sph}(X)$ for $\Res_{\sph}(\Delta)$ and $\CCC_{\sph}(X)$ for $\CCC_{\sph}(\Delta)$.

For any spherical residue $x\in\Res_{\sph}(\Delta)$ and any sequence $(R_n)_{n\in\NN}$ of spherical residues converging to some $\xi\in\CCC_{\sph}(\Delta)$, define the {\bf combinatorial sector} $Q(x,\xi)$ pointing towards $\xi$ and based at $x$ as \[Q(x,\xi):=\bigcup_{k\geq 0}{\bigcap_{n\geq k}{\Conv(x,R_n)}},\] where $\Conv(x,R_n)$ denotes the convex hull in $\Res_{\sph}(\Delta)$ of the pair of residues $\{x,R_n\}$. This definition turns out to be indeed independent of the sequence $(R_n)$ converging to $\xi$. Moreover, each such sector is contained in some apartment of $\Delta$ (\cite[2.3]{CL11}). 

Finally, we recall from \cite[5.1]{CL11} that one can associate to every point $\xi$ in the visual boundary $\partial X$ of $X$ a {\bf transversal building} $\Delta^{\xi}$ of dimension strictly smaller than $\dim X$, on which the stabiliser $G_{\xi}$ of $\xi$ acts by type-preserving simplicial isometries.

\subsection*{Almost connected locally compact groups}
Before we can proceed with the proof of the results announced in the introduction, we need one more technical result. 

\begin{lemma}\label{lemma positive measure}
Let $G$ be an almost connected locally compact group. Then every measurable subgroup $H$ of $G$ of positive measure has finite index in $G$. 
\end{lemma}
\begin{proof}
Since $G$ is compactly generated, it is in particular $\sigma$-compact. So there exists a compact subset $K\subseteq G$ such that $H\cap K$ has positive (finite) measure. Then by \cite[Cor.20.17]{HewittRoss}, there is an open neighbourhood $U$ of the identity such that $U\subseteq (H\cap K)(H\cap K)\inv\subseteq H$, so that $H$ is open. Hence $H$ contains the connected component $G^0$ of $G$. Since moreover $G/G^0$ is compact and the natural projection $\pi\co G\to G/G^0$ is open, $\pi(H)$ has finite index in $G/G^0$, whence the lemma.
\end{proof}

\begin{remark}\label{rk positive measure}
Note that Lemma~\ref{lemma positive measure} ensures that, given an almost connected locally compact group $G$, the $G$-actions on trees without fixed point constructed by Serre as in the introduction (see Remark~\ref{rkintro Serre}) are not measurable. Indeed, if $G$ is the countable union of a strictly increasing sequence of subgroups, then one of them must be non-measurable. For otherwise one of them would have positive measure by $\sigma$-additivity, and hence would have finite index in $G$ by Lemma~\ref{lemma positive measure}, a contradiction.
\end{remark}

We also record for future reference the following structure result for connected locally compact groups, which follows from the solution to Hilbert's fifth problem (see \cite[Thm.4.6]{Montgomzippinhilbert5}).

\begin{lemma}\label{lemma Hilbert5}
Let $G$ be a connected locally compact group. Then there is a compact normal subgroup $N$ of $G$ such that $G/N$ is a connected Lie group. 
\end{lemma}

The key result needed for the proof of Theorem \ref{mainthm} is now the following.

\begin{theorem}\label{thm K} 
Let $G$ be either a compact abelian group or the group $(\RR,+)$. Then $G$ has property $\FB$.
\end{theorem}

\begin{proof}
Let $\Delta$ be a finite rank building on which $G$ acts measurably.
We prove that $G$ stabilises some spherical residue of $\Delta$ by induction on the dimension of $\Delta$. If $\Delta$ has zero dimension, then it is spherical and there is nothing to prove. Assume now that $\Delta$ has positive dimension. Let $X$ denote its Davis realization. We first need to know that each element of $G$ fixes some point of $X$.

\medskip \noindent
\underline{Claim 1}: \emph{The action of $G$ on $X$ is locally elliptic.}

\medskip \noindent
This follows from \cite[Thm 2.5]{CM11} in case $G$ is compact, and from \cite[Proof of Thm 2.5, Claim 7]{CM11} in case $G=\RR$ since $\RR$ is divisible abelian. 

\medskip \noindent
For a subset $F\subset G$, let $X^F$ denote the set of $F$-fixed points in $X$. Note that each $X^F=\bigcap_{g\in F}{\Min(g)}$ is a closed convex subset of $X$.

\medskip \noindent
\underline{Claim 2}: \emph{For each finite subset $F\subset G$, the set $X^F$ is non-empty.}

\medskip \noindent
Indeed, by Claim~1, each set $X^{\{g\}}$ for some $g\in G$ is non-empty. Let $F\subset G$ be finite and suppose that $X^F$ is non-empty. Let $g\in G$ and $x\in \Min(g)$. Since $G$ is abelian, $g$ stabilises $X^F$ and therefore fixes the projection of $x$ on $X^F$, so that $X^{F\cup\{g\}}$ is still non-empty (see for example \cite[Proof of Thm 2.5, Claim 4]{CM11}). The claim then follows by a straightforward induction.

\medskip \noindent
Since the subsets $X^F$ of $X$ for finite $F\subset G$ form a filtering family of non-empty closed convex subsets of $X$ by Claim~2, it follows from Lemma~\ref{lemma CLytchack} that either $\bigcap{X^F}$ is non-empty, where the intersection runs over all finite $F\subset G$, in which case the induction step stands proven, or the corresponding intersection of the visual boundaries $\bigcap{\partial X^F}$ is a non-empty subset of $\partial X$. We may thus assume that the group $G$ fixes some $\xi\in\partial X$. We now prove that $G$ already fixes a point in $X$. 
Let $X^{\xi}$ denote the transversal building to $X$ associated to $\xi$. Thus, $G$ acts on $X^{\xi}$.

\medskip \noindent
\underline{Claim 3}: \emph{Let $H$ be a subgroup of $G$. Suppose $H$ fixes some point $\zeta\in\CCC_{\sph}(X)$. Let $x\in\Res_{\sph}(X)$. Then every element of $H$ fixes pointwise a subsector of $Q(x,\zeta)$.}

\medskip \noindent
Indeed, let $h\in H$. Then Claim~1 yields a spherical residue $x_h\in\Res_{\sph}(X)$ which is fixed by $h$. It follows from \cite[Lem.4.4]{CL11} that $h$ fixes the combinatorial sector $Q(x_h,\zeta)$ pointwise. Since by \cite[Prop.2.30]{CL11} there is some $z_h\in \Res_{\sph}(X)$ such that $Q(z_h,\zeta)\subset Q(x,\zeta)\cap Q(x_h,\zeta)$, the conclusion follows.

\medskip \noindent
\underline{Claim 4}: \emph{The action of $G$ on $X^{\xi}$ is measurable.}

\medskip \noindent
Indeed, we have to check that the stabiliser $H$ in $G$ of a spherical residue of $X^{\xi}$ is measurable. Since $\Res_{\sph}(X^{\xi})\subseteq \CCC_{\sph}(X^{\xi})$ can be identified with a subset of $\CCC_{\sph}(X)$ by \cite[Thm 5.5]{CL11}, we may assume that $H$ is the stabiliser in $G$ of a point $\zeta\in\CCC_{\sph}(X)$. Let $x\in\Res_{\sph}(X)$, and for each spherical residue $y\in Q(x,\zeta)$, let $H_y$ denote the pointwise fixator in $G$ of the combinatorial sector $Q(y,\zeta)$. Note that in fact $H_y\leq H$ for all $y$. It follows from Claim~3 that $H$ is the union of all such $H_y$. Since $Q(x,\zeta)$ lies in some apartment and since apartments possess only countably many spherical residues, this union is countable. Thus, it is sufficient to check that the pointwise fixator in $G$ of a combinatorial sector $Q(y,\zeta)$ is measurable. Since this fixator is the (again countable) intersection of the stabilisers of the spherical residues in $Q(y,\zeta)$, the claim follows since such stabilisers are measurable by hypothesis. 

\medskip \noindent
It follows from Claim~4 and from the induction hypothesis that $G$ stabilises some spherical residue $\zeta\in \Res_{\sph}(X^{\xi})\subseteq \CCC_{\sph}(X^{\xi})$. Again, we may identify $\CCC_{\sph}(X^{\xi})$ with a subset of $\CCC_{\sph}(X)$, so that $G$ stabilises some point in $\CCC_{\sph}(X)$, again denoted by $\zeta$. Then as before, Claim~3 implies that $G$ is covered by countably many stabilisers of points of $X$. Since these are measurable, one of them, say $G_x$ for some $x\in X$, must have positive measure by $\sigma$-additivity. Then $G_x$ has finite index in $G$ by Lemma~\ref{lemma positive measure}. We can now complete the induction step using Lemma~\ref{lemma reduction}.
\end{proof}

\begin{corollary}\label{cor fixed point}
Let $G$ be an almost connected locally compact group acting measurably on a finite rank building $\Delta$ with Davis realization $X$. Assume that $G$ fixes a point in the combinatorial bordification of $\Delta$. Then $H$ already fixes a point in $X$.
\end{corollary}
\begin{proof}
This is what we have just established using Lemma~\ref{lemma positive measure} to conclude the proof of Theorem~\ref{thm K}.
\end{proof}

The following result is probably well known; since we could not find it explicitly stated in the published literature, we include it here with a complete proof.

\begin{theorem}\label{thm Lie}
Let $G$ be a connected Lie group. Then $G$ is a bounded product of one-parameter subgroups.
\end{theorem}
\begin{proof}
Note first that $G$ decomposes as a product of a maximal connected compact subgroup and of finitely many one-parameter subgroups (see for example \cite{Borelcpct}). We may thus assume that $G$ is compact connected. Let $x_1,\dots,x_n$ be a basis of the Lie algebra $\Lie(G)$, and for each $i\in\{1,\dots,n\}$, let $U_i=\overline{\exp(\RR x_i)}$ be the closure in $G$ of the one-parameter subgroup associated to $x_i$. Since $G$ is compact, each $U_i$ is compact and hence so is the bounded product $A=U_1\dots U_n$. In particular, $A$ is closed. Since $G$ is connected, it is generated by $A$, so that $G=\bigcup_{n\geq 1}{A^n}$. Now, by Baire's theorem, there exists an $n\in\NN$ such that $A^n$ has non-empty interior, and so $A^{2n}$ contains an open neighbourhood $U$ of the identity in $G$. Since $G$ is compact, there is a finite subset $F\subset G$ such that $G=FU$. Then $F\subset A^k$ for some $k\in\NN$ and so $G=A^{2n+k}$. Finally, note that each $U_i$ is a connected compact abelian Lie group, hence a torus (see for example \cite[Prop.2.42(ii)]{Compactlie}). Since clearly each torus is a bounded product of one-parameter subgroups, the conclusion follows.
\end{proof}

\begin{remark}\label{rk compact fg}
Note that the argument above in fact yields the following: \emph{A compact group that is generated by finitely many abelian subgroups has finite abelian width.}
\end{remark}

\begin{remark}\label{rk Lie faw}
It follows from Theorem~\ref{thm Lie} that any almost connected Lie group $G$ has finite abelian width. Indeed, since the connected component of the identity of a Lie group is open, $G$ is virtually connected. The claim then follows since $G^0$ has finite abelian width by Theorem~\ref{thm Lie}.
\end{remark}

\noindent
{\bf Proof of Corollary~\ref{corintro Lie}.} Let $G$ be an almost connected Lie group. Then $G$ is virtually connected since $G^0$ is open. By Lemma~\ref{lemma reduction} we may thus assume that $G$ is connected. Then by Theorem~\ref{thm Lie}, it is a bounded product of one-parameter subgroups. The conclusion then follows from Theorem~\ref{thm K}, together with Lemma~\ref{lemma reduction}. \hspace{\fill}$\Box$

\medskip
\noindent 
{\bf Proof of Corollary~\ref{corintro compact}.} The first statement is an immediate consequence of Theorem~\ref{thm K} together with Lemma~\ref{lemma reduction}. Then, since compact $p$-adic analytic groups are finitely generated by \cite[Cor.8.34]{padicanalytic}, they have property $\FB$ by Remark~\ref{rk compact fg}. Finally, profinite groups of polynomial subgroup growth also have finite abelian width by the main result of \cite{Bddgenprofinite}. \hspace{\fill}$\Box$


\section{Some remarks and the proof of Theorem~\ref{mainthm} concluded}\label{section rk cor}

In this section, we explain why Conjecture~\ref{conj tree} implies Conjecture~\ref{conj measurable}. Since Conjecture~\ref{conj tree} is known in case the building is a tree by a result of Tits, this will imply Theorem~\ref{thm tree}. The same argument yields the proof of Theorem~\ref{mainthm}, except that in this case Conjecture~\ref{conj tree} is not available and we use the assumption of finite abelian width instead.

\begin{remark}\label{remark FBP}
Let $G$ be either a compact group of finite abelian width or an almost connected Lie group. Then in fact $G$ satisfies a slightly stronger property than property $\FB$, which is the following {\bf property $\FBP$}: 

\begin{itemize}
\item[$\FBP$]
\emph{For each finite rank building $\Delta$ with Davis realization $X$, for each subgroup $Q\leq\Aut_0(X)$ having a global fixed point in $X$, any measurable action of $G$ on the $\CAT$ subcomplex $\Min(Q)\leq X$ by type-preserving cellular isometries has a global fixed point.}
\end{itemize}

Note that in case $Q$ is trivial this is just property $\FB$. To see that $G$ indeed satisfies this property $\FBP$ when $G$ is either compact abelian or the group $(\RR,+)$, remark first that Claim~1 and Claim~2 from the proof of Theorem~\ref{thm K} remain valid in this context, as well as Lemma~\ref{lemma CLytchack}. Moreover, setting $Y:=\Min(Q)\subseteq X$, we can identify the visual boundary $\partial Y$ with a subset of $\partial X$. Then Lemma~\ref{lemma CLytchack} either yields the desired conclusion or yields a fixed point $\xi\in\partial Y\subseteq \partial X$ for the $G$-action. Then $G$ acts on the closed convex subset $Y_1:=\Min(Q)\subseteq X^{\xi}$ of the transversal building $X^{\xi}$ for the induced action of $Q$ on $X^{\xi}$. Moreover, by \cite[Lem.4.4]{CL11}, each combinatorial sector $Q(x,\zeta)$ for some $x\in Y$ and some $\zeta\in Y_1\subset\CCC_{\sph}(X)$ is entirely contained in $Y$. Then the proofs of Claim~3 and Claim~4 go through without any change and we may apply the induction hypothesis to find a global fixed point for $G$ on $Y$. 

Finally, the case where $G$ is a compact group of finite abelian width or an almost connected Lie group follows from Lemma~\ref{lemma reduction}.
\end{remark}

We summarize this in the following lemma.
\begin{lemma}\label{lemma FBP}
Compact groups of finite abelian width and almost connected Lie groups have property $\FBP$.
\end{lemma}

The interest of this slightly more general fixed point property is that it allows to construct new examples of groups with property $\FB$ starting from known examples.

\begin{lemma}\label{lemma extension}
Let $G$ be a locally compact group and let $N\triangleleft G$ be a closed normal subgroup of $G$ such that $G/N$ has property $\FBP$. If $N$ has property $\FB$, then so has $G$.
\end{lemma}
\begin{proof}
Suppose $G$ acts measurably on a building $\Delta$ with Davis realization $X$. By hypothesis, $Y:=\Min(N)\subseteq X$ is non-empty and stabilised by $G$. Moreover, the $G$-action on $Y$ coincides with the induced $G/N$-action on $Y$, which is still measurable. Thus $G/N$ fixes a point in $Y$ by hypothesis. The conclusion follows.
\end{proof}

\noindent
{\bf Proof that Conjecture~\ref{conj tree} implies Conjecture~\ref{conj measurable}.} Let $G$ be an almost connected locally compact group acting measurably by type-preserving simplicial isometries on a building $\Delta$ with Davis realization $X$. Assume that Conjecture~\ref{conj tree} holds. We have to prove that $G$ fixes a point in $X$.

\medskip \noindent
\underline{Claim 1}: \emph{Let $H$ be a group acting locally elliptically on $X$. Then $H$ fixes a point in the combinatorial bordification of $X$.}

\medskip \noindent
This follows from a straightforward induction on $\dim X$ using Conjecture~\ref{conj tree} and \cite[Thm 5.5]{CL11}.

\medskip \noindent
\underline{Claim 2}: \emph{Let $H$ be an almost connected locally compact group acting measurably and locally elliptically on $X$. Then $H$ already fixes a point in $X$.}

\medskip \noindent
Since $H$ fixes a point in the combinatorial bordification of $X$ by Claim~1, the claim follows from Corollary~\ref{cor fixed point}.

\medskip 

Let $N$ be a compact normal subgroup of the connected component $G^0$ of $G$ such that $G^0/N$ is a connected Lie group (see Lemma~\ref{lemma Hilbert5}). 

By \cite[Thm 2.5]{CM11}, we know that $N$ acts locally elliptically on $X$, and hence fixes a point in $X$ by Claim 2. Consider now the induced action of the connected Lie group $G^0/N$ on the fixed point set $\Min(N)$ of $N$ in $X$. Since $G^0/N$ has property $\FBP$ by Lemma~\ref{lemma FBP}, it fixes a point in $\Min(N)$. This shows that $G^0$ fixes a point of $X$. 

In turn, one can consider the action of the compact group $G/G^0$ on $\Min(G^0)\subseteq X$. This is a locally elliptic action by \cite[Thm 2.5]{CM11}, and hence the $G$-action on $X$ is locally elliptic. Thus Claim 2 yields the desired fixed point for the $G$-action on $X$. \hspace{\fill}$\Box$

\begin{remark}\label{rk iterated transversal}
Note that if, given a specific building $\Delta$, we want to show that any almost connected locally compact group $G$ acting measurably by type-preserving simplicial isometries on $\Delta$ has a global fixed point, it is sufficient to check that Conjecture~\ref{conj tree} holds for group actions on $\Delta$, as well as on the ``iterated transversal buildings'' of $\Delta$ (see the proof of Claim~1 above). Here, we mean by ``iterated transversal buildings'' of $\Delta$ either the transversal buildings to $\Delta$, or the transversal buildings to these transversal buildings, and so on.  

In particular, if Conjecture~\ref{conj tree} holds for group actions on trees, then so does Theorem~\ref{thm tree}. More generally, remark that if $\Delta$ is of type $(W,S)$, then the type of a ``iterated transversal building'' of $\Delta$ is a subgroup of $W$. So, for example, if Conjecture~\ref{conj tree} holds for group actions on Euclidean (respectively, right-angled) buildings, then Conjecture~\ref{conj measurable} also holds for the same class of buildings.
\end{remark}

\medskip
\noindent
{\bf Proof of Theorem~\ref{thm tree}.} As mentioned in Remark~\ref{rk iterated transversal} above, this follows from the fact that Conjecture~\ref{conj tree} is true when the building is a tree (see for example \cite[6.5 Exercise 2]{Serre}). \hspace{\fill}$\Box$

\medskip
\noindent
{\bf Proof of Theorem~\ref{mainthm}.}
Let $G$ be an almost connected locally compact group of finite abelian width. Thus $G$ is a bounded product of abelian subgroups, which we may assume to be closed, hence also locally compact. Therefore, by Lemma~\ref{lemma reduction}(1), it is sufficient to prove the theorem when $G$ is abelian, which we assume henceforth.

Let $G^0$ denote the connected component of $G$. By Lemma~\ref{lemma Hilbert5} we know that there exists some compact normal subgroup $N$ of $G^0$ such that $G^0/N$ is a connected Lie group. Since by assumption $G^0$ is abelian, we know by Lemma~\ref{lemma FBP} that $N$ and $G^0/N$ satisfy property $\FBP$, and hence that $G^0$ has property $\FB$ by Lemma~\ref{lemma extension}. Finally, since $G/G^0$ is compact abelian, the same argument yields that $G$ has property $\FB$, as desired.
\hspace{\fill}$\Box$

\begin{remark}
Note that, by the same argument as above, Conjecture~\ref{conj measurable} now reduces to proving that compact groups have property $\FBP$. Since every compact group is a product of its identity component with a totally disconnected closed subgroup (\cite[Thm.9.41]{Compactlie}), Lemma~\ref{lemma reduction} in turn reduces the problem to compact connected and profinite groups. We now mention some consequences of Lemma~\ref{lemma FBP} in each of these two cases.

If $G$ is compact connected, then there exists a family $\{S_j \ | \ j\in J\}$ of simple simply connected compact Lie groups such that $G$ is a quotient of $\prod_{j\in J}{S_j}\times Z_0(G)$ by a closed central subgroup, where $Z_0(G)$ denotes the identity component of the center of $G$ (\cite[Thm.9.24]{Compactlie}). Thus, in that case, Lemma~\ref{lemma reduction} reduces the problem to compact connected groups of the form $\prod_{j\in J}{S_j}$. Note that the simple simply connected compact Lie groups are classified and belong to countably many isomorphism classes. Index these classes by $\NN$. Then one can write $G$ as a countable product $G=\prod_{i\in \NN}{T_i}$, where $T_i$ is the product of all $S_j$ belonging to the $i$-th isomorphism class. Moreover, each $T_i$ has property $\FBP$. Indeed, let $S$ be a representative for the $i$-th class, so that $T_i$ is isomorphic to a product $\prod_{j\in I}{S}$ of copies of $S$. By Theorem~\ref{thm Lie}, the group $S$ is boundedly generated by compact abelian subgroups $A_1,\dots,A_n$. Hence $T_i$ is boundedly generated by the compact abelian subgroups $\prod_{j\in I}{A_1},\dots,\prod_{j\in I}{A_n}$, whence the claim. Thus, if only finitely many isomorphism classes appear in the product decomposition of $G$, then $G$ has property $\FBP$.

If $G$ is a profinite group, then $G$ is a closed subgroup of a product $\prod_{j\in J}{S_j}$ of finite groups. Suppose that $G=\prod_{j\in J}{S_j}$. Then, as in the case of compact connected groups, we may express $G$ as a countable product $G=\prod_{i\in \NN}{T_i}$, where $T_i$ is this time the product of all $S_j$ of order $i$. Again, the same argument shows that each $T_i$ has property $\FBP$, since clearly a group of order $i$ is boundedly generated by $i$ abelian subgroups. 
Finally, as noted in Corollary~\ref{corintro compact}, profinite groups of polynomial subgroup growth also have finite abelian width and thus property $\FBP$.
\end{remark}


\section{One-parameter subgroups of real Kac--Moody groups}\label{section cor KM}

In this section, we prove Theorem~\ref{thmintro KM} and Corollaries~\ref{cor KMR}, \ref{corintro KM2} and \ref{corintro KM3} from the introduction. 

\medskip
\noindent
{\bf Proof of Theorem~\ref{thmintro KM}.}
Since a measurable homomorphism from an almost connected Lie group $H$ into a Kac--Moody group $G$ yields measurable actions of $H$ on the associated positive and negative buildings of $G$, the result readily follows from Corollary~\ref{corintro Lie}. \hspace{\fill} $\Box$

\medskip
\noindent
Let now $k$ be either $\RR$ or $\CC$, and let $G=\GGG(k)$ be a real or complex adjoint Kac--Moody group of finite rank (see \cite{Tits87}).  
The group $G$ comes equipped with a root group datum $\{U_{\alpha} \ | \ \alpha\in\Phi^{\re}\}$ indexed by the real roots of $G$, and we assume that these root groups generate $G$. Write $I=\{1,\dots,\rk(G)\}$. Then, choosing a root basis $\{\alpha_i \ | \ i\in I\}$, each subgroup of $G$ of the form $X_i:=\langle U_{-\alpha_i},U_{\alpha_i}\rangle$ is a copy of $\SL(2,k)$, which we endow with its natural topology. We next endow $G$ with the so-called Kac--Peterson topology (\cite{Kacpetersontopo}), which is the finest topology such that for all $n\in \NN$ and all $i_1,\dots,i_n\in I$, the multiplication map \[X_{i_1}\times\dots\times X_{i_n}\to G: (x_1,\dots,x_n)\mapsto x_1\dots x_n\] is continuous, where each product $X_{i_1}\times\dots\times X_{i_n}$ has the product topology. This turns $G$ into a connected Hausdorff topological group (\cite[4G, p.163]{Kacpetersontopo2}, \cite[Prop.5.15, Rk.5.16]{HartKohlmars}). Moreover, each $X_i$ is closed in $G$, and the induced topology from $G$ coincides with the natural topology of $\SL(2,k)$. Note however that $G$ is not locally compact.

\begin{lemma}\label{lemma parabolics closed}
The finite type parabolic subgroups of $G$ of each sign are closed in $G$. In particular, the $\sigma$-algebra on $G$ generated by these parabolics is contained in the Borel algebra of $G$.
\end{lemma}
\begin{proof}
Following \cite[6.2.2]{theseBR}, a finite type parabolic subgroup $P$ of $G$ has a Levi decomposition $P=L\ltimes U$, with $L$ the Levi factor and $U$ the unipotent radical of $P$. Note that $L$ is closed in $G$ since it is locally compact by \cite[Rk.5.13]{HartKohlmars}. Moreover, by a straightforward adaptation of the proof of \cite[Prop.5.11]{HartKohlmars}, the subgroup $U$ is closed as well. Since $P$ is the topological product of $L$ and $U$, the claim follows.

We remark that in the $2$-spherical situation this is in fact an immediate consequence of \cite[Thm.1 and Prop.3.20]{HartKohlmars}.
\end{proof}

Remember that Corollary~\ref{cor KMR} is well-known for finite-dimensional Lie groups (see e.g. \cite[Prop.2.10]{HofMorliegr}).

\begin{lemma}\label{lemma cor KM Lie}
Let $H$ be a Lie group. Then every $\alpha\in\Hom_c(\RR,H)$ is of the form $\alpha(t)=\exp\ad(tx)$ for some $x\in\Lie(H)$.
\end{lemma}

Denote by ${\bf X}=(X_{+},X_{-})$ the twin building associated to $G$.

\begin{lemma}\label{lemma cor KM CM}
Let $x_+$ and $x_-$ be spherical residues in $X_+$ and $X_-$, respectively. Then the stabiliser $G_{x_+,x_-}$ in $G$ of these two residues is closed in $G$. Moreover, with the induced topology, it has the structure of a finite-dimensional almost connected Lie group $H\cong G_{x_+,x_-}$, and every vector $x\in\Lie(H)$ can be identified with an $\ad$-locally finite vector of $\Lie(G)$.
\end{lemma}
\begin{proof}
It follows from \cite[Prop.3.6]{CM06} that $G_{x_+,x_-}$ possesses a Levi decomposition $G_{x_+,x_-}=L\ltimes U$ with Levi factor $L$ and with unipotent radical $U$ associated to parallel residues in $X_+$ and $X_-$ (see \cite[3.2.1]{CM06}). We show that $L$ and $U$ are both almost connected Lie groups, whence the structure of almost connected Lie group on $G_{x_+,x_-}$. 

By \cite[Lem.3.3]{CM06} together with \cite[Cor.3.5]{CM06}, the subgroup $U$ is also the unipotent radical associated to a pair of chambers in $X_+$ and $X_-$. It follows by \cite[Lem.3.2]{CM06} that $U$ is a bounded product of finitely many root groups, which carry the Lie group topology by \cite[Cor.5.12]{HartKohlmars}. In particular, $U$ is connected. Moreover, the Lie algebra of $U$ is the direct sum of the finitely many Lie algebras of the root groups which boundedly generate $U$, and is therefore finite-dimensional. Hence $U$ is a Lie group, as desired.

The Levi factor $L$ is a Lie group because of \cite[Cor.5.12 and Rk.5.13]{HartKohlmars}. It is almost connected, since it decomposes as a product of a torus $T$ (homeomorphic to $(k^*)^{\rk(G)}$) with a subgroup generated by root groups.

The claim about $\ad$-locally finiteness can be found in \cite[Theorems 1 and 2]{Kacpeterson3}.
\end{proof}

\noindent
{\bf Proof of Corollary~\ref{cor KMR}.}
Let $\alpha\in\Hom_c(\RR,G)$. By Theorem~\ref{thmintro KM} together with Lemma~\ref{lemma parabolics closed}, we know that $\alpha(\RR)$ is bounded in $G$. The conclusion then follows from Lemmas~\ref{lemma cor KM Lie} and \ref{lemma cor KM CM}. \hspace{\fill} $\Box$

\medskip
\noindent
{\bf Proof of Corollary~\ref{corintro KM2}.} 
Let $\alpha\co G_1\to G_2$ be a measurable homomorphism between adjoint real or complex Kac--Moody groups $G_1, G_2$. Note first that by Theorem~\ref{thmintro KM} together with Lemmas~\ref{lemma parabolics closed} and \ref{lemma cor KM CM}, the image of any measurable homomorphism $\beta\co\SL(2,k)\to G_2$ is contained in a Lie group. In particular such a $\beta$ must be continuous by \cite[Theorem 1]{automcont}. It follows that for any copies $X_{i_1},\dots,X_{i_n}$ of $\SL(2,k)$ in $G_1$, the map $\overline{\alpha}$ in the following commutative diagram is continuous:

\begin{equation*}\begin{CD}
X_{i_1}\times\dots\times X_{i_n} @>>> G_1\\
@V\overline{\alpha}VV @VV\alpha V\\
\alpha(X_{i_1})\times\dots\times \alpha(X_{i_n}) @>>> G_2.
\end{CD}\end{equation*}
\noindent
Continuity of $\alpha$ then follows by definition of the Kac--Peterson topology. \hspace{\fill} $\Box$

\medskip
\noindent
{\bf Proof of Corollary~\ref{corintro KM3}.}
Let $\alpha\co G_1\to G_2$ be a measurable isomorphism between adjoint real or complex Kac--Moody groups $G_1, G_2$. Recall from \cite[Thm.6.3]{CM06} that $\alpha$ is standard whenever it maps bounded subgroups to bounded subgroups. Let thus $H$ be a bounded subgroup of $G_1$. Then $H$ is contained in an almost connected Lie group by Lemma~\ref{lemma cor KM CM}. Thus $\alpha(H)$ is bounded by Theorem~\ref{thmintro KM} together with Lemma~\ref{lemma parabolics closed}, as desired.  \hspace{\fill} $\Box$

\bibliographystyle{amsalpha} 
\bibliography{paramKMRbib}

\end{document}